\documentclass[11pt,final]{article}

\usepackage{textcomp}
\usepackage{amsmath,amsthm,amssymb,amscd}
\usepackage{color}

\numberwithin{equation}{section}

\theoremstyle{definition}\newtheorem{definition}{Definition}
\theoremstyle{plain}\newtheorem{theorem}[definition]{Theorem}
\theoremstyle{plain}
\theoremstyle{plain}\newtheorem{corollary}[definition]{Corollary}
\theoremstyle{plain}\newtheorem{lemma}[definition]{Lemma}
\theoremstyle{definition}\newtheorem{assumption}[definition]{Assumption}
\theoremstyle{definition}
\theoremstyle{definition}
\theoremstyle{definition}

\newcommand{\cO}{{\mathcal O}}

\newcommand{\diff}{{\,\mathrm{d}}}

\newcommand{\re}{\mathrm{Re}\,}
\newcommand{\bbR}{\mathbb{R}}
\newcommand{\bbN}{\mathbb{N}}
\newcommand{\la}{\langle}
\newcommand{\ra}{\rangle}

\newcommand{\dist}{\mathrm{dist}}

\begin{document}

\title{There always is a variational source condition for nonlinear problems in Banach spaces}

\author{
\textsc{Jens Flemming}%
\footnote{Chemnitz University of Technology,
Faculty of Mathematics, D-09107 Chemnitz, Germany,
jens.flemming@mathematik.tu-chemnitz.de.}
}

\date{\today
}

\maketitle

\begin{abstract}
Variational source conditions proved useful for deriving convergence rates for Tikhonov's regularization method and also for other methods.
Up to now such conditions have been verified only for few examples or for situations which can be handled by classical techniques, too.
\par
Here we show that for almost every ill-posed inverse problem variational source conditions are satisfied.
Whether linear or nonlinear, whether Hilbert or Banach spaces, whether one or multiple solutions, variational source conditions are a universal tool for proving convergence rates.
\end{abstract}

\section{Setting}

We want to solve nonlinear equations
\begin{equation}\label{eq:the}
F(x)=y^\dagger,\quad x\in X,
\end{equation}
in Banach spaces $X$ and $Y$, where $F:D(F)\rightarrow Y$ has domain $D(F)\subseteq X$.
The exact right-hand side $y^\dagger$ may be know only approximately as a noisy measurement $y^\delta\in Y$ with
\begin{equation*}
\|y^\delta-y^\dagger\|\leq\delta
\end{equation*}
for some noise level $\delta\geq 0$.
In addition, the solutions need not depend continuously on this data.

Regularization is required and we concentrate on Tikhonov's method
\begin{equation}\label{eq:tikh}
\|F(x)-y^\delta\|^p+\alpha\,\Omega(x)\to\min_{x\in D(F)}.
\end{equation}
We assume $p\geq 1$ and $\alpha>0$.
The penalty $\Omega:X\rightarrow(-\infty,\infty]$ is allowed to attain $+\infty$ as a value.
This method is well understood and existence, stability and convergence of minimizers can be guaranteed by the following assumptions, see \cite[Section~3.2]{SchGraGroHalLen09} or \cite[Chapter~4]{SchKalHofKaz12}.

\begin{assumption}\label{as:basic}
We assume that the following properties are satisfied by the introduced setting.
\begin{itemize}
\item[(i)]
Equation~\eqref{eq:the} has a solution with finite $\Omega$.
\item[(ii)]
$F$ is weakly sequentially closed, that is, each sequence $(x_n)_{n\in\bbN}$ in $D(F)$ satisfies
\begin{equation*}
x_n\rightharpoonup x,\quad F(x_n)\rightharpoonup y\qquad\Rightarrow\qquad x\in D(F),\quad F(x)=y.
\end{equation*}
\item[(iii)]
$\Omega$ is convex.
\item[(iv)]
The sublevel sets $\{x\in X:\Omega(x)\leq c\}$, $c\in\bbR$, are weakly sequentially closed and each sequence in such a set has a weakly convergent subsequence.
\end{itemize}
\end{assumption}

Tikhonov regularized solutions always converge, at least in a subsequential manner, to solutions which minimize the penalty $\Omega$ in the set of all solutions. $\Omega$-minimizing solutions are typically denoted by $x^\dagger$.
Convergence to such $\Omega$-minimizing solutions may be arbitrarily slow and we are interested in estimates for the convergence speed.
Of course, additional assumptions are required for such estimates. This issue will be discussed in the next section.

At first, we have to decide how to measure the speed of convergence. The canonical choice is the norm distance between regularized and exact solution, but alternatives like the Bregman distance (see below) proved useful, too. Also point-to-set distances could be used if there are multiple exact solutions. To cover all these situations we introduce a general error functional \mbox{$E^\dagger:X\rightarrow[0,\infty)$}, where the symbol $\dagger$ indicates that the error functional depends on one or more $\Omega$-minimizing solutions.

Denoting by $x_\alpha^\delta$ the minimizers of the Tikhonov minimization problem \eqref{eq:tikh}, we aim at asymptotic estimates
\begin{equation}\label{eq:rates}
E^\dagger(x_\alpha^\delta)=\cO(\varphi(\delta)),\qquad\delta\to 0,
\end{equation}
where $\alpha$ may depend on $\delta$ and $y^\delta$.
The function $\varphi$ shall be an index function in the following sense.

\begin{definition}
A function $\varphi:[0,\infty)\rightarrow[0,\infty)$ is an \emph{index function} if it is continuous, monotonically increasing, strictly increasing in a neighborhood of zero, and satisfies $\varphi(0)=0$.
\end{definition}

\section{Variational source conditions revisited}

Different techniques have been developed to prove convergence rates \eqref{eq:rates}.
The most prominent tool are source conditions for linear ill-posed inverse problems in Hilbert spaces. The classical concept is described in \cite[Section~3.2]{EngHanNeu96} and general source conditions are studied in \cite{MatPer03}. See also the references given in \cite{MatPer03} for the origins of general source conditions.  In both cases the norm distance between exact and regularized solution is used as error functional $E^\dagger$.

For Banach spaces usage of source conditions is quite limited. But in 2007 variational source conditions were introduced in \cite{HofKalPoeSch07} and thoroughly studied and developed during the past 10 years, see, e.\,g., \cite{BotHof10,Fle12,Gra10b,HohWer13}.
This type of condition allows to prove convergence rates for many different settings, especially for nonlinear operators and general penalty functionals in \eqref{eq:tikh}. In its original version Bregman distances (see Section~\ref{sc:special}) were used as error functional $E^\dagger$.

Variational source conditions are also known as variational inequalities, but this term conflicts with the already existing mathematical field with the same name.
A second alternative was introduced in the book \cite{Fle12}. There the term variational smoothness assumption is used, because several kinds of smoothness (not only of the underlying exact solution as it is the case for classical source conditions) are jointly described by one expression.
The term variational source condition rouses associations to classical source conditions. But the new concept has no similarity to classical source conditions, most notably there is no source element.
Nevertheless, in most recent literature `variational source condition' seems to be used more often than `variational inequality', whereas `variational smoothness assumption' is not used by other authors. Thus, we write `variational source condition' to name the technique described and applied below and, to avoid drawing too many parallels to source conditions, we read it as `variational replacement for source conditions'.

The definition of variational source conditions in the present paper will be slightly more general than other variants before, because it is not connected to one fixed solution of \eqref{eq:the}. Instead, we allow multiple $\Omega$-mini\-mizing solutions and refer to \cite{BueFleHof16} for a concrete example. For this purpose we denote by $\Omega^\dagger$ the value of $\Omega$ at the $\Omega$-minimizing solutions, that is,
\begin{equation*}
\Omega^\dagger:=\min\{\Omega(x^\dagger):\,x^\dagger\in D(F),\,F(x^\dagger)=y^\dagger\}.
\end{equation*}
Assumption~\ref{as:basic} guarantees existence of $\Omega$-minimizing solutions.

\begin{definition}
Let $\beta>0$ be a constant and let $\varphi:[0,\infty)\rightarrow[0,\infty)$ be an index function. A \emph{variational source condition} for fixed right-hand side $y^\dagger$ holds on a set $M\subseteq D(F)$ if
\begin{equation}\label{eq:vsc}
\beta\,E^\dagger(x)\leq\Omega(x)-\Omega^\dagger+\varphi(\|F(x)-y^\dagger\|)\quad\text{for all $x\in M$}.
\end{equation}
\end{definition}

If the set $M$ is large enough to contain all minimizers of the Tikhonov functional \eqref{eq:tikh}, then a variational source condition \eqref{eq:vsc} implies the desired convergence rate \eqref{eq:rates}. Although our variant is slightly more general, the proofs of this fact given in \cite[Chapter~4]{Fle12} or in \cite{HofMat12} still work with trivial modifications. Suitable choices of the regularization parameter $\alpha$ are discussed there, too.

The constant $\beta$ plays only a minor role. In principle we could hide it in the functional $E^\dagger$, but then $E^\dagger$ would depend on the chosen index function $\varphi$ and not solely on exact and regularized solutions. The implied convergence rate does not depend on $\beta$, only the $\cO$-constant contains the factor $\frac{1}{\beta}$.

Variational source conditions originally were developed to obtain rates for Tikhonov regularization, but can also be used in the context of other methods. See \cite{GraHalSch11} for the residual method and \cite{HohWer13} for iteratively regularized Newton methods.

A major drawback of variational source conditions is that the best obtainable rate may be slower than the best possible one. This is for instance the case for rates faster than $\cO(\sqrt{\delta})$ in the classical linear Hilbert space setting, where the best one is $\cO(\delta^{\frac{2}{3}})$. On the other hand, in $\ell^1$-regularization rates up to the best possible one $\cO(\delta)$ for the error norm can be obtained, see \cite{BurFleHof13}.
An approach to overcome technical rate limitations was undertaken in \cite{Gra13}, but it is limited to linear equations.

\section{Main result}

The aim of this paper is to answer the question under which conditions we find a variational source condition of type \eqref{eq:vsc}.
For linear equations in Hilbert spaces we know that general source conditions can always be satisfied, see \cite{MatHof08}, and that general source conditions imply variational source conditions without reducing the implied convergence rate, see \cite[Chapter~13]{Fle12}.
In Banach spaces up to now variational source conditions were only verified for few concrete examples, see, e.\,g., \cite{BueFleHof16,HofKalPoeSch07,KoeHohWer16}.
The only exception is $\ell^1$-regularization for linear equations, where in \cite{FleGer17} it was shown that one always finds a variational source condition for some index function $\varphi$.

Of course, we have to connect the error functional $E^\dagger$ in some way to the other ingredients of a variational source condition.

\begin{assumption}\label{as:E}
Given $\beta>0$ and $M\subseteq D(F)$ we assume the following.
\begin{itemize}
\item[(i)]
$M$ is weakly sequentially closed.
\item[(ii)]
All solutions $x^\ast$ to \eqref{eq:the} satisfy the inequality in \eqref{eq:vsc} if they belong to $M$, that is,
\begin{equation*}
\beta\,E^\dagger(x^\ast)\leq\Omega(x^\ast)-\Omega^\dagger
\end{equation*}
for all $x^\ast\in M$ with $F(x^\ast)=y^\dagger$.
\item[(iii)]
The mapping
\begin{equation*}
x\mapsto-\beta\,E^\dagger(x)+\Omega(x)
\end{equation*}
is weakly sequentially lower semicontinuous on $M$.
\item[(iv)]
There are constants $\tilde{\beta}>\beta$ and $\tilde{c}\geq 0$ such that
\begin{equation*}
\tilde{\beta}\,E^\dagger(x)-\Omega(x)\leq\tilde{c}\quad\text{for all $x\in M$}.
\end{equation*}
\end{itemize}
\end{assumption}

These assumptions will be verified for several important special cases in Section~\ref{sc:special}.
There we will see that all relevant settings are covered.

\begin{theorem}\label{th:main}
Let Assumption~\ref{as:basic} and Assumption~\ref{as:E} be true for a constant $\beta>0$ and a set $M\subseteq D(F)$.
Then there exists a concave index function $\varphi$ such that the variational source condition \eqref{eq:vsc} is satisfied.
\end{theorem}

The assertion of the theorem is quite similar to the main result of \cite{MatHof08}, where it is shown that in linear Hilbert space settings one always finds an index function such that a corresponding general source condition is satisfied.
Our theorem extends this results to nonlinear Banach space settings.

The proof of the theorem relies on the technique of approximate variational source conditions introduced in \cite{FleHof10} and thoroughly studied in \cite[Chapter~12]{Fle12}. Here we only introduce the parts we need for the proof.
The idea is to measure the violation of a variational source condition with linear index function $\varphi$ by the \emph{distance function}
\begin{equation}\label{eq:D}
D_\beta(r):=\sup_{x\in M}\left(\beta\,E^\dagger(x)-\Omega(x)+\Omega^\dagger-r\,\|F(x)-y^\dagger\|\right)
\end{equation}
for $r\geq 0$.
This function is the supremum of affine functions and thus convex and continuous on the interior of its domain. Obviously, it is monotonically decreasing.

\begin{lemma}\label{th:Dphi}
Let $D_\beta$ be defined by \eqref{eq:D}.
If $D_\beta(0)>0$ and
\begin{equation*}
\lim_{r\to\infty}D_\beta(r)=0,
\end{equation*}
then
\begin{equation*}
\varphi(t)=\inf_{r\geq 0}\bigl(D_\beta(r)+r\,t\bigr),\quad t\geq 0,
\end{equation*}
defines a concave index function and a variational source condition \eqref{eq:vsc} with this $\varphi$ and same $\beta$ and $M$ as in \eqref{eq:D} holds true.
\end{lemma}

\begin{proof}
From
\begin{align*}
\lefteqn{\beta\,E^\dagger(x)-\Omega(x)+\Omega^\dagger}\\
&\qquad=\inf_{r\geq 0}\bigl(\beta\,E^\dagger(x)-\Omega(x)+\Omega^\dagger-r\,\|F(x)-y^\dagger\|+r\,\|F(x)-y^\dagger\|\bigr)\\
&\qquad\leq\inf_{r\geq 0}\bigl(D_\beta(r)+r\,\|F(x)-y^\dagger\|\bigr)
\end{align*}
for all $x\in M$ we obtain the asserted variational source condition if $\varphi$ is indeed an index function.
\par
Since $D_\beta$ is decreasing and goes to zero, it has to be nonnegative. Consequently, $0\leq\varphi(t)<\infty$. We also immediately see that $\varphi$ is monotonically increasing. In addition, $\varphi$ is concave and upper semicontinuous as an infimum of affine functions.
Thus, $\varphi$ is continuous on the interior $(0,\infty)$ of its domain.
The decay of $D_\beta$ to zero yields $\varphi(0)=0$, which, together with upper semicontinuity, yields continuity on the whole domain $[0,\infty)$.
The assumption $D_\beta(0)>0$ ensures $\varphi(t)>0$ for some $t$. Thus, $\varphi$ has to be strictly increasing near zero.
\end{proof}

\begin{proof}[Proof of Theorem~\ref{th:main}]
We want to apply Lemma~\ref{th:Dphi}. If $D_\beta(0)\leq 0$, then the variational source condition holds for arbitrary index functions $\varphi$. So we may assume $D_\beta(0)>0$ and the lemma reduces the proof to verification of $D_\beta(r)\to 0$ if $r\to\infty$.
In addition, we may assume $D_\beta(r)\geq 0$ for all $r>0$, because $D_\beta(r)<0$ for some $r$ would imply a variational source condition with the best possible concave index function $\varphi(t)=r\,t$.
\par
As first step we show that for fixed $r\geq 0$ the supremum in the definition \eqref{eq:D} of $D_\beta(r)$ is attained.
We write
\begin{equation}\label{eq:Dinf}
D_\beta(r):=-\inf_{x\in M}\left(-\beta\,E^\dagger(x)+\Omega(x)-\Omega^\dagger+r\,\|F(x)-y^\dagger\|\right)
\end{equation}
and denote by $(x_n)_{n\in\bbN}$ an infimizing sequence. The functional in the infimum is bounded on the sequence by some constant $c$.
With $\tilde{\beta}$ and $\tilde{c}$ as in Assumption~\ref{as:E}~(iv) we see
\begin{align*}
-\beta\,E^\dagger(x_n)+\Omega(x_n)
&=\frac{\beta}{\tilde{\beta}}\,\bigl(-\tilde{\beta}\,E^\dagger(x_n)+\Omega(x_n)\bigr)
+\left(1-\frac{\beta}{\tilde{\beta}}\right)\,\Omega(x_n)\\
&\geq\frac{\beta\,\tilde{c}}{\tilde{\beta}}+\left(1-\frac{\beta}{\tilde{\beta}}\right)\,\Omega(x_n),
\end{align*}
which implies
\begin{align*}
\left(1-\frac{\beta}{\tilde{\beta}}\right)\,\Omega(x_n)
&\leq-\beta\,E^\dagger(x_n)+\Omega(x_n)-\frac{\beta\,\tilde{c}}{\tilde{\beta}}\\
&\leq-\beta\,E^\dagger(x_n)+\Omega(x_n)-\Omega^\dagger+r\,\|F(x_n)-y^\dagger\|+\Omega^\dagger-\frac{\beta\,\tilde{c}}{\tilde{\beta}}\\
&\leq c+\Omega^\dagger-\frac{\beta\,\tilde{c}}{\tilde{\beta}}.
\end{align*}
Thus, $(\Omega(x_n))_{n\in\bbN}$ is bounded and we find a subsequence of $(x_n)_{n\in\bbN}$ converging weakly to some $\tilde{x}\in X$. The subsequence will be denoted again by $(x_n)_{n\in\bbN}$. Assumption~\ref{as:E}~(i) ensures $\tilde{x}\in M$ and with item (iii) one easily shows that $\tilde{x}$ is a minimizer in \eqref{eq:Dinf}.
\par
As second step we take a sequence $(r_n)_{n\in\bbN}$ in $[0,\infty)$ with $r_n\to\infty$ and corresponding maximizers $x_n$ in the definition \eqref{eq:D} of $D_\beta(r_n)$ to show $D_\beta(r)\to\infty$ if $r\to\infty$.
As discussed above, Assumption~\ref{as:E}~(iv) implies
\begin{align*}
\left(1-\frac{\beta}{\tilde{\beta}}\right)\,\Omega(x_n)
&\leq-\beta\,E^\dagger(x_n)+\Omega(x_n)-\Omega^\dagger+r_n\,\|F(x_n)-y^\dagger\|+\Omega^\dagger-\frac{\beta\,\tilde{c}}{\tilde{\beta}}\\
&=-D_\beta(r_n)+\Omega^\dagger-\frac{\beta\,\tilde{c}}{\tilde{\beta}}\\
&\leq\Omega^\dagger-\frac{\beta\,\tilde{c}}{\tilde{\beta}}.
\end{align*}
Thus, $(\Omega(x_n))_{n\in\bbN}$ is bounded and we find a subsequence of $(x_n)_{n\in\bbN}$ converging weakly to some $\tilde{x}\in M$. The subsequence will be denoted again by $(x_n)_{n\in\bbN}$ with corresponding $(r_n)_{n\in\bbN}$.
\par
From $D_\beta(r_n)\geq 0$ we obtain
\begin{equation*}
r_n\,\|F(x_n)-y^\dagger\|\leq\beta\,E^\dagger(x_n)-\Omega(x_n)+\Omega^\dagger
\end{equation*}
and the right-hand side is bounded by Assumption~\ref{as:E}~(iv).
Consequently, $r_n\to\infty$ implies $\|F(x_n)-y^\dagger\|\to 0$ and the lower semicontinuity of $x\mapsto\|F(x)-y^\dagger\|$ yields $F(\tilde{x})=y^\dagger$, that is, the maximizers in the definition of $D_\beta$ converge (subsequentially) to solutions of \eqref{eq:the}.
If we combine this observation with items (ii) and (iii) in Assumption~\ref{as:E}, we obtain
\begin{align*}
0
&\leq\liminf_{n\to\infty}D_\beta(r_n)
\leq\limsup_{n\to\infty}D_\beta(r_n)
=-\liminf_{n\to\infty}\bigl(-D_\beta(r_n)\bigr)\\
&=-\liminf_{n\to\infty}\bigl(-\beta\,E^\dagger(x_n)+\Omega(x_n)-\Omega^\dagger+r_n\,\|F(x_n)-y^\dagger\|\bigr)\\
&\leq-\liminf_{n\to\infty}\bigl(-\beta\,E^\dagger(x_n)+\Omega(x_n)-\Omega^\dagger\bigr)\\
&\leq-\bigl(-\beta\,E^\dagger(\tilde{x})+\Omega(\tilde{x})-\Omega^\dagger\bigr)\\
&\leq 0.
\end{align*}
This proves $D_\beta(r_n)\to 0$ and, since $(r_n)_{n\in\bbN}$ was chosen arbitrarily, also $D_\beta(r)\to 0$ if $r\to\infty$.
The fact that $(r_n)_{n\in\bbN}$ is only a subsequence of the original sequence causes no troubles, because $D_\beta$ is monotonically decreasing.
Thus, $D_\beta(r_n)\to 0$ has to hold for the original sequence, too.
\end{proof}

Theorem~\ref{th:main} states that there is always an index function $\varphi$ for a variational source condition.
In principle the proof is constructive, but calculating the distance function $D_\beta$ is a difficult task.
From the proof of Lemma~\ref{th:Dphi} we see that also a majorant for $D_\beta$ yields an index function $\varphi$ as long as this majorant decays to zero at infinity.
Index functions $\varphi$ constructed in this spirit can be found in \cite{BueFleHof16,FleGer17} for special settings.
A very similar approach to obtain index functions $\varphi$ is discussed in \cite{HohWei17}.

\section{Special cases}\label{sc:special}

We provide some special cases for which Theorem~\ref{th:main} is applicable.

\subsection{Linear equations in Hilbert spaces}

Let $X$ and $Y$ be Hilbert spaces and let $A:=F:X\rightarrow Y$ be linear and bounded.
For the Tikhonov functional \eqref{eq:tikh} we choose $p=2$ and $\Omega(x)=\|x\|^2$.
For each $\alpha$ there is exactly one Tikhonov minimizer.
Equation~\eqref{eq:the} may have multiple solutions, but there is exactly one $\Omega$-minimizing solution, which we denote by $x^\dagger$.
As error functional $E^\dagger$ we choose
\begin{equation*}
E^\dagger(x)=\|x-x^\dagger\|^2,\quad x\in X.
\end{equation*}

\begin{corollary}
For each $\beta\in(0,1)$ there exists a concave index function $\varphi$ such that the variational source condition \eqref{eq:vsc} is satisfied with $M=X$.
\end{corollary}

\begin{proof}
Assumption~\ref{as:basic} is obviously true and we only have to check Assumption~\ref{as:E} to apply Theorem~\ref{th:main}.
Item (i) is trivially true. Item (ii) reads
\begin{equation*}
\beta\,\|x^\ast-x^\dagger\|^2\leq\|x^\ast\|^2-\|x^\dagger\|^2\quad\text{for all $x^\ast\in X$ with $A\,x^\ast=y^\dagger$.}
\end{equation*}
For solutions $x^\ast$ the difference $x^\ast-x^\dagger$ is in the nullspace of $A$ and $x^\dagger$ is in the orthogonal complement of the null space. Thus, the inner product $\la x^\dagger,x^\ast-x^\dagger\ra$ vanishes and we obtain the desired estimate
\begin{align*}
\beta\,\|x^\ast-x^\dagger\|^2
&=\beta\,\bigl(\|x^\ast\|^2-\|x^\dagger\|^2-2\,\re\la x^\dagger,x^\ast-x^\dagger\ra\bigr)\\
&=\beta\,\bigl(\|x^\ast\|^2-\|x^\dagger\|^2\bigr)\\
&\leq\|x^\ast\|^2-\|x^\dagger\|^2.
\end{align*}
Concerning item (iii) we observe
\begin{equation*}
-\beta\,\|x-x^\dagger\|^2+\|x\|^2
=(1-\beta)\,\|x\|^2-2\,\beta\,\re\la x,x^\dagger\ra+\beta\,\|x^\dagger\|^2
\end{equation*}
for all $x\in X$.
Since $\beta<1$ by assumption, the functional is obviously weakly sequentially lower semicontinuous.
Finally, item (iv) follows from
\begin{align*}
\tilde{\beta}\,\|x-x^\dagger\|-\|x\|^2
&=-(1-\tilde{\beta})\,\|x\|^2-2\,\tilde{\beta}\,\re\la x,x^\dagger\ra+\tilde{\beta}\|x^\dagger\|^2\\
&\leq-(1-\tilde{\beta})\,\|x\|^2+2\,\tilde{\beta}\,\|x\|\,\|x^\dagger\|+\tilde{\beta}\|x^\dagger\|^2,
\end{align*}
because the right-hand side is bounded for all $\tilde{\beta}<1$. So we may choose $\tilde{\beta}\in(\beta,1)$.
\end{proof}

The fact that in a linear Hilbert space setting there is always a variational source condition has been proven already in a different and more complicated way in \cite[Chapter~13]{Fle12} based on the result that there is always a general source condition (cf.\ \cite{MatHof08}).
In \cite[Section~13.2]{Fle12} it is shown that the constant $\beta$ in a variational source condition does not influence the index function $\varphi$ in a linear Hilbert space setting, up to some scaling.

\subsection{Bregman distance in Banach spaces}

A first convergence rate result for Tikhonov regularization \eqref{eq:tikh} with convex penalty $\Omega$ can be found in \cite{BurOsh04} based on a source condition. There the error functional $E^\dagger$ is a \emph{Bregman distance} with respect to $\Omega$, which is defined as follows.
Let $x^\dagger$ be an $\Omega$-minimizing solution with nonempty subdifferential and denote by $\xi^\dagger\in X^\ast$ a subgradient of $\Omega$ at $x^\dagger$.
Then
\begin{equation}\label{eq:bregman}
B_{\xi^\dagger}(x,x^\dagger):=\Omega(x)-\Omega(x^\dagger)-\la\xi^\dagger,x-x^\dagger\ra_{X^\ast\times X},\quad x\in X,
\end{equation}
is the Bregman distance between $x$ and $x^\dagger$.
If $\Omega$ is a Hilbert space norm, then the corresponding Bregman distance is the usual norm distance.

Bregman distances became a standard tool for convergence rate analysis in Banach spaces.
The original variational source condition in \cite{HofKalPoeSch07} used them, too.
With Theorem~\ref{th:main} we can prove the following.

\begin{corollary}
Let Assumption~\ref{as:basic} be true and assume that $x^\dagger$ is an $\Omega$-mini\-mizing solution with corresponding subgradient $\xi^\dagger$.
Further let $E^\dagger$ be the Bregman distance \eqref{eq:bregman}.
If there are no other solutions to \eqref{eq:the}, then there exist a constant $\beta\in(0,1)$ and a concave index function $\varphi$ such that the variational source condition \eqref{eq:vsc} is satisfied with $M=D(F)$.
\end{corollary}

\begin{proof}
We have to show that Assumption~\ref{as:E} is true.
Item (i) is a consequence of the weak closedness of $F$.
Item (ii) is trivially true, because there is only one solution.
Weak lower semicontinuity of
\begin{equation*}
x\mapsto-\beta\,E^\dagger(x)+\Omega(x)=(1-\beta)\,\Omega(x)+\beta\,\Omega^\dagger+\beta\,\la\xi^\dagger,x-x^\dagger\ra_{X^\ast\times X}
\end{equation*}
follows from Assumption~\ref{as:basic}~(iv) and $\beta\leq 1$.
\par
It remains to show boundedness of $x\mapsto\tilde{\beta}\,E^\dagger(x)-\Omega(x)$ for some $\tilde{\beta}>\beta$.
Weak sequential compactness of the sublevel sets of $\Omega$ implies boundedness of the sublevel sets. Thus, by \cite[Exercise~2.41 and pages 324--326]{Zal02} there are a positive constant $c_1$ and a constant $c_2$ such that
\begin{equation*}
\Omega(x)\geq c_1\,\|x\|+c_2\qquad\text{for all $x\in X$.}
\end{equation*}
With this observation and with $\tilde{\beta}\in(\beta,1]$ we obtain
\begin{align*}
\lefteqn{\tilde{\beta}\,E^\dagger(x)-\Omega(x)}\\
&\qquad=(\tilde{\beta}-1)\,\Omega(x)-\tilde{\beta}\,\Omega^\dagger-\tilde{\beta}\,\la\xi^\dagger,x-x^\dagger\ra_{X^\ast\times X}\\
&\qquad\leq(\tilde{\beta}-1)\,(c_1\,\|x\|+c_2)-\tilde{\beta}\,\Omega^\dagger+\tilde{\beta}\,\|\xi^\dagger\|\,\|x\|+\tilde{\beta}\la\xi^\dagger,x^\dagger\ra_{X^\ast\times X}\\
&\qquad=\bigl((c_1+\|\xi^\dagger\|)\,\tilde{\beta}-c_1\bigr)\,\|x\|+(\tilde{\beta}-1)\,c_2-\tilde{\beta}\,\Omega^\dagger+\tilde{\beta}\la\xi^\dagger,x^\dagger\ra_{X^\ast\times X}
\end{align*}
The last expression is bounded with respect to $x\in X$ if
\begin{equation*}
\tilde{\beta}\leq\frac{c_1}{c_1+\|\xi^\dagger\|}.
\end{equation*}
Thus, Assumption~\ref{as:E} is true for all $\beta\in(0,\frac{c_1}{c_1+\|\xi^\dagger\|})$.
\end{proof}

\subsection{Multiple solutions in autoconvolution}

In \cite{BueFleHof16} a variational source condition for complex-valued autoconvolution
\begin{equation*}
(F(x))(s)=\int_{\max\{s-1,0\}}^{\min\{s,1\}}x(s-t)\,x(t)\diff t,\quad s\in(0,2),
\end{equation*}
with $X=L^2(0,1)$ and $Y=L^2(0,2)$ has been verified, if the considered exact solutions have a sparse Fourier representation.

If $x^\dagger$ is a solution to \eqref{eq:the}, then $-x^\dagger$ is a second solution and there are only these two solutions.
We choose $\Omega$ in \eqref{eq:tikh} to be the $L^2$-norm.
There is no reason to prefer one of the two solutions and the Tikhonov minimizers will converge (subsequentially) to both solutions.
Thus, for $E^\dagger$ we choose the point-to-set distance
\begin{equation}\label{eq:autoE}
E^\dagger(x)=\dist\bigl(x,\{x^\dagger,-x^\dagger\}\bigr)^2=\min\bigl\{\|x-x^\dagger\|^2,\|x+x^\dagger\|^2\bigr\},\quad x\in X.
\end{equation}

The variational source condition obtained in \cite{BueFleHof16} holds only on a set $M$ which is the union of two small balls around $x^\dagger$ and $-x^\dagger$. The new Theorem~\ref{th:main} yields a variational source condition with $M=X$, but possibly with a different index function $\varphi$.

\begin{corollary}
If $E^\dagger$ is given by \eqref{eq:autoE}, then there are a constant $\beta\in(0,1)$ and a concave index function $\varphi$ such that the variational source condition \eqref{eq:vsc} is satisfied with $M=X$.
\end{corollary}

\begin{proof}
Assumption~\ref{as:basic} is obviously satisfied, see \cite{AnzBueHofSte16} for weak lower semicontinuity of $F$.
Items (i) and (ii) of Assumption~\ref{as:E} are trivially true.
For item (iii) we observe
\begin{equation*}
E^\dagger(x)=\|x\|^2+\|x^\dagger\|^2-2\,|\re\la x,x^\dagger\ra|
\end{equation*}
and therefore
\begin{equation*}
-\beta\,E^\dagger(x)+\Omega(x)
=(1-\beta)\,\|x\|^2-\beta\,\|x^\dagger\|^2+2\,\beta\,|\re\la x,x^\dagger\ra|,
\end{equation*}
which is a lower semicontinuous functional.
Item (iv) follows from
\begin{align*}
\tilde{\beta}\,E^\dagger(x)-\Omega(x)
&=-(1-\tilde{\beta})\,\|x\|^2-2\,\tilde{\beta}\,|\re\la x,x^\dagger\ra|+\tilde{\beta}\,\|x^\dagger\|^2\\
&\leq-(1-\tilde{\beta})\,\|x\|^2+2\,\tilde{\beta}\,\|x\|\,\|x^\dagger\|+\tilde{\beta}\,\|x^\dagger\|^2
\end{align*}
for $\tilde{\beta}\in(\beta,1)$, because the last expression is bounded with respect to \mbox{$x\in X$}.
\end{proof}

In \cite{BueFleHof16} a relatively strong assumption (sparse Fourier representation) was required to obtain a variational source condition and corresponding convergence rates. With the corollary above we now have a variational source condition and rates without additional assumptions.
Although we do not know $\varphi$ explicitly, variational source conditions turn out to be the right tool for convergence rate analysis.

\subsection{$\ell^1$-regularization}

If $X=\ell^1(\bbN)$ and if $\Omega$ in \eqref{eq:tikh} is the $\ell^1$-norm, then Tikhonov regularization is also known as $\ell^1$-regularization.
In \cite{FleGer17} it was shown, that for injective bounded linear operators in \eqref{eq:the} there is always a variational source condition and the index function $\varphi$ can be made explicit. There, the $\ell^1$-norm distance between exact and regularized solutions was used as error functional $E^\dagger$. This norm distance does not coincide with the Bregman distance with respect to the $\ell^1$-norm.

For noninjective bounded linear operators variational source conditions and corresponding convergence rates were derived in \cite{Fle16}, but under additional assumptions. The proofs there are quite technical, but the interested reader easily verifies that Theorem~\ref{th:main} is applicable to the noninjective $\ell^1$-setting, too, if the error functional $E^\dagger$ is the point-to-set distance between some $x\in\ell^1(\bbN)$ and the set of all $\ell^1$-norm minimizing solutions of \eqref{eq:the}. Consequently, even without injectivity we always obtain convergence rates for $\ell^1$-regularization.
We do not provide a proof here, because we would have to go deep into the technicalities of \cite{Fle16}.

\bibliography{vsc}

\end{document}